\documentclass{amsart}

\usepackage{amsmath}

\usepackage{amsfonts}

\usepackage{hyperref}

\usepackage{latexsym}

\usepackage{amssymb}

\usepackage{amscd}

\newtheorem{theorem}{Theorem}[section]
\newtheorem{lemma}{Lemma}[section]
\newtheorem{corollary}{Corollary}[section]
\newtheorem{proposition}{Proposition}[section]
\newtheorem{definition}{Definition}[section]

\newtheorem{remark}{Remark}[section]

\newtheorem{example}{Example}[section]

\newcommand{\ma}[1]{\ensuremath{\mathbb{#1}}}
 
\font\bb=msbm7 at 10 pt

\def \Z {\hbox{\bb Z}}
\def \Q {\hbox{\bb Q}}
\def \N {\hbox{\bb N}}

\def \F {\hbox{\bb F}}
\def \R {\hbox{\bb R}}
\def \A {\hbox{\bb A}}

\def \Tr {\mbox{\rm{Tr}}}
\def \d {\mbox{\bf{d}}}
\def \v {\mbox{\bf{v}}}

\def \AA {\mathcal{A}}

\def \L {\mathcal{L}}

\def \Q {\hbox{\bb Q}}

\def \M {\mathcal{M}}

\newcommand{\Spec}{\ensuremath{\mbox{\rm{Spec }}}}

\author{R\'egis Blache}
\address{\'Equipe AOC,
IUFM de la Guadeloupe}
\email{blache@iufm.univ-ag.fr}

\title[$p$-density and applications]{$p$-density, exponential sums and Artin-Schreier curves}

\begin{document}

\begin{abstract}
In this paper we define the $p$-density of a finite subset $D\subset\ma{N}^r$, and show that it gives a good lower bound for the $p$-adic valuation of exponential sums over finite fields of characteristic $p$. We also give an application: when $r=1$, the $p$-density is the first slope of the generic Newton polygon of the family of Artin-Schreier curves associated to polynomials with their exponents in $D$.
\end{abstract}




\maketitle

\section{Introduction} 

This paper deals with the $p$-adic valuation of exponential sums over a finite field of characteristic $p$. A classical result in this field is Stickelberger's theorem on the valuation of Gauss sums \cite{st}. It is closely related to the question of the existence of rational points on an algebraic variety over a finite field. This question was raised by E. Artin, then solved by Chevalley, and precised by Warning \cite{war}, leading to the celebrated Chevalley-Warning theorem. During the 60s, the emergence of Dwork's ideas allowed improvements on this result, first by Ax \cite{ax}, then by Katz \cite{ka1}. Both based their result on a suitable estimation for the $p$-adic valuation of exponential sums, and their bounds only depends on the degrees of the polynomials defining the variety. More recently, Adolphson and Sperber have improved these results, considering more precisely the monomials appearing in the polynomial \cite{as}. They look at its Newton polyhedron, i.e. the convex closure of the exponents effectively appearing in it.

\medskip

All these bounds are independent of the prime $p$, and are optimal in the following sense: if we fix degrees (resp. a Newton polyhedron), then for any $p$ we can find a (system of) polynomial(s) such that the valuation of the associated exponential sum meets the bound.

\medskip

In the 90s, Moreno and Moreno took into account the prime $p$. Using a reduction to the ground fied method, they replaced the degrees of the polynomials by their $p$-weights (the sums of their base $p$ digits) and found a new bound, improving the existing ones in many cases \cite{mm1}. 
Recently, in a joint work with Castro, Kumar and Shum, O. Moreno reduced the problem of estimating the $p$-adic valuation of an exponential sum to the one of estimating the minimal $p$-weight of the set of solutions a system of modular equations \cite{mm2}. They shown that their bound is tight. Note that it must be finer than the preceding ones. It is the first one really taking into account only the monomials appearing: when we use a Newton polyhedron, we consider a convex hull, and perhaps ``add" monomials which were not originally in the polynomial we were considering.

\medskip

Let us be more precise: for a prime $p$, and some of its powers $q$, let $f\in \F_q[x_1,\dots,x_r]$ be a polynomial, and let $D$ the set of exponents in $\N^r$ of the monomials of $f$. The system introduced by Moreno et al. in order to give a bound for the $p$-adic valuation of the exponential sum asociated to $f$ over $\F_q$ involves equations modulo $q-1$ whose coefficients are the coordinates of the elements of $D$. Moreno et al. give lower bounds for the $p$-weights of the solutions, which often give improvements on the previously known bounds. 

\medskip

In this paper we define the $p$-density of a finite subset $D$ of $\N^r$. It is a lower bound for the $p$-weights of the solutions of the Moreno et al. system, depending only on the prime $p$ and the set $D$. It is tight in the sense that there exists infinitely many powers $q$ of $p$ such that it is attained for the modulus $q-1$. As a consequence, it allows to give uniform bounds on the valuation (depending only on $p$ and $D$). Moreover we are able to bound the minimal power for which it is optimal.

\medskip

Apart from the Chevalley-Warning type theorems, these bounds have many applications. In this paper we give one to the generic first slope of Artin Schreier curves with $p$-rank zero ({\it cf} \cite{sz1},\cite{sz2},\cite{sz3}). Recall that these are curves defined over $\F_q$ with an affine equation of the form
$$y^p-y=f(x),~f\in \F_q[x],$$
i.e. $p$-cyclic coverings of the projective line ramified only at infinity. We show that when $f$ runs over monic polynomials with their exponents in a fixed $D\subset \N$, then the first slope of the generic Newton polygon is exactly the $p$-density of $D$. This provides a unified and simple method to treat the first slope question, at least in the generic case. This result also allows to prove some families of Artin Schreier curves of small genus to be supersingular in a straightforward way. 

\medskip

Note that in some sense these results give a generalisation of Stickelberger's theorem: when the set $D$ consists of one integer $d$, the $p$-density is the least valuation for a Gauss sum over an extension of $\F_p$ associated to a multiplicative character of order dividing $d$. This is coherent with the result about Artin Schreier curves, since for such $D$ we consider the curve $y^p-y=x^d$, and the slopes of the Newton polygon of the numerator of its zeta function are exactly the valuations of these Gauss sums.
 
 \medskip
 
We believe that the bounds given here have many other applications, to classical mathematical problems such as Waring problem over finite fields or Serre-Weil bounds, also in information theory where exponential sums are useful for estimating certain invariants of codes, boolean functions, etc... 

\medskip

The paper is organised as follows: in section 1 we define the $p$-density of a finite subset $D\subset \N^r$, and give some of its relevant properties. We give the main theorem on the valuation of exponential sums in section 2, and a Chevalley-Warning-Ax-Katz theorem. In the last section, we deal with Artin-Schreier curves.

\medskip

\section{$p$-density}

In this section, we fix a prime $p$, and a finite, nonempty subset $D=\{\d_i\}_{1\leq i\leq n}$, $\d_i=(d_{i1},\dots,d_{ir})$, of $\N^r\backslash\{0,\dots,0\}$. We assume that $D$ is not contained in any of the sets $x_j=0$, $1\leq j\leq r$; in this case we just have to lower the dimension $r$ to apply the results below to $D$.

For $n$ a non negative integer, we denote by $\sigma_p(n)$ the $p$-weight of $n$ : in other words, if $n=n_0+pn_1+\dots+p^tn_t$ with $0\leq n_i\leq p-1$, we have $\sigma_p(n)=n_0+\dots+n_t$.

\medskip

\subsection{Modular equations and the definition of $p$-density} Here we shall make precise the objects we work along the first section. We give some elementary properties, in order to define the notion of $p$-density.

\begin{definition}
Let $D$ be as above, and $m$ denote a positive integer. 

i/ We define $E_D(m)$ as the set of $n$-tuples $U=(u_1,\dots,u_n)\in \{0,\dots,p^m-1\}^n\backslash \{0,\dots,0\}$ such that
$$\sum_{i=1}^n u_i\d_i \equiv 0~[p^m-1],~\sum_{i=1}^n u_id_{ij}>0,~\mbox{\rm for all}~1\leq j\leq r.$$
For any $U \in E_D(m)$, we define the {\rm $p$-weight of $U$} as the integer $\sigma_p(U)= \sum_{i=1}^n \sigma_p(u_i)$, and the {\rm length of $U$} as $\ell(U)=m$.

ii/ Define $\sigma_p(D,m):=\min_{U\in E_D(m)} \sigma_p(U)$.

iii/ Let $\delta_m$ be the {\rm shift}, from the set $\{0,\dots,p^m-1\}$ to itself, which sends any integer $0\leq n\leq p^m-2$ to the residue of $pn$ modulo $p^m-1$, and $p^m-1$ to itself. We extend it coordinatewise to the set $\{0,\dots,p^m-1\}^n$.

iv/ We define a map
$$\begin{array}{ccccc}
\varphi & : & E_D(m) & \rightarrow & \N^r \\
        &   & U      & \mapsto & \frac{1}{p^m-1} \sum_{i=1}^n u_i\d_i\\
        \end{array}$$
        
v/ For any $U \in E_D(m)$, we set $\Phi(U):=\{\varphi(\delta_m^k(U)),~0,\leq k\leq m-1\}$.
\end{definition}  

\begin{remark} i/ Note that the set $E_D(m)$ is finite, thus the number $\sigma_p(D,m)$ is well defined.

ii/ The map $\delta_m$ shifts the $p$-digits of the integer $n$, whence its name. As a consequence, it preserves the $p$-weight. On the other hand we have the congruence $\delta(U)\equiv pU ~[p^m-1]$. Moreover we have $\delta_m^{m}=Id$ : the $m$-th iterate of the map $\delta_m$ is the identity. Finally for any $n\in \{0,\dots,p^m-1\}$ we have the equality
$$\sum_{k=0}^{m-1} \delta_m^k(n)=\frac{p^m-1}{p-1}\sigma_p(n).$$

iii/ We assume in the following that for any $1\leq i\leq n$ the prime $p$ does not divide $\d_i$. Actually if $\d_i=p\d_i'$, the map sending $U=(u_1,\dots,u_i,\dots,u_n)$ to $U'=(u_1,\dots,\delta_m(u_i),\dots,u_n)$ is a bijection from $E_D(m)$ to $E_{D'}(m)$, where $D'=(\d_1,\dots,\d_i',\dots,\d_n)$, and it preserves the weight.

iv/ We are working in the set $\N^r$ (resp. $\N^n$); in order to simplify the notations, we consider it as a subset of the $\Z$-module $\Z^r$ (resp. $\Z^n$), and note the laws as usual in these modules. 

\end{remark}

We begin by giving some easy properties of the maps we have just defined. We do not prove give a proof since they come readily from the definitions above.

\begin{lemma} 
\label{fp}
Let $D$, $m$ be as above. For any $1\leq j\leq r$, let $D_j=\sum_{i=1}^n d_{ij}$ be the total degree of $D$ along its $j$-th coordinate.

i/ The map $\delta$ sends $E_D(m)$ to itself, and preserves the $p$-weight.

ii/ For any $U=(u_1,\dots,u_n)\in \{0,\dots,p^m-1\}^n$, we have 
$$\sum_{k=0}^{m-1} \delta_m^k(U)=\frac{p^m-1}{p-1}(\sigma_p(u_1),\dots,\sigma_p(u_n)).$$

iii/ The image of $\varphi$ is contained in $\prod_{j=1}^r\{1,\dots,D_j\}$. 
\end{lemma}

We now give some other consequences of the definition, which we shall use in order to define the $p$-density of the set $D$.

\begin{lemma} 

\label{sp}
Let $D$, $m$ be as above, and $U=(u_1,\dots,u_n)\in E_D(m)$. Choose an integer $1\leq t\leq m-1$, and for any $1\leq i\leq n$ let $u_i=p^t w_i+v_i$ be the euclidean division of $u_i$ by $p^t$.

i/ We have the equality 
$$\sum_{i=1}^{n} \sigma_p(u_i)\d_i=(p-1)\sum_{k=0}^{m-1} \varphi(\delta_m^k(U)).$$

ii/ For any $t$ as above, we have the equalities :
$$\sum_{i=1}^{n} v_i\d_i=p^t\varphi(\delta_m^{-t}(U))-\varphi(U)~;~\sum_{i=1}^{n} w_i\d_i=p^{m-t}\varphi(U)-\varphi(\delta_m^{-t}(U)).$$

\end{lemma}

\begin{proof}

Part i/ is an easy consequence of the definitions : from lemma \ref{fp} ii/, we have $\sum_{i=1}^n \sum_{k=0}^{m-1} \delta_m^k(u_i)\d_i=\frac{q-1}{p-1}\sum_{i=1}^n \sigma_p(u_i)\d_i$. On the other hand, $\sum_{i=1}^n \delta_m^k(u_i)\d_i=(q-1)\varphi(\delta_m^k(U))$ from the definition of the map $\varphi$.

We come to part ii/. Let $u_i=\sum_{k=0}^{m-1} u_{ik} p^k$ where $0\leq u_{ik}\leq p-1$. For $U$ as above, and $0\leq k\leq m-1$, define $U_k:=(u_{1k},\dots,u_{nk})$. An easy calculation shows that for any $i$ we have $pu_i-\delta(u_i)=(q-1)u_{i,m-1}$, and 
$$(q-1)\left(p\varphi(U)-\varphi(\delta_m(U))\right)=p\sum_{i=1}^n u_i\d_i-\sum_{i=1}^n \delta(u_i)\d_i=(q-1)\sum_{i=1}^n u_{i,m-1}\d_i.$$
That is : $\sum_{i=1}^n u_{i,m-1}\d_i= p\varphi(U)-\varphi(\delta_m(U))$. Now from its definition, we have $w_i=\sum_{k=t}^{m-1} u_{ik}p^{k-t}$. Therefore we get
$$\begin{array}{ccl}
\sum_{i=1}^n w_i\d_i & = & \sum_{k=t}^{m-1} p^{k-t} \sum_{i=1}^n u_{ik}\d_i \\
& = & \sum_{k=t}^{m-1} p^{k-t}\left(p\varphi(\delta_m^{m-1-k}(U))-\varphi(\delta_m^{m-k}(U))\right)\\
& = & p^{m-t}\varphi(U)-\varphi(\delta_m^{-t}(U))\\
\end{array}$$
The proof for the $v_i$ follows the same lines.
\end{proof}

\medskip

We are ready to show the main result of this section 

\medskip

\begin{proposition}
\label{princ}
The set $\left\{\frac{\sigma_p(D,m)}{m}\right\}_{m\geq 1}$ has a minimum; this minimum is attained for at least one $m\leq \prod_{j=1}^r D_i$.
\end{proposition}

\begin{proof}
Let $m>\prod_{j=1}^r D_i$ be an integer, and choose $U=(u_1,\dots,u_n)\in E_D(m)$ such that $\sigma_p(U)=\sigma_p(D,m)$. Recall that we have defined $\Phi(U):=\{\varphi(\delta_m^k(U))\}_{0\leq k\leq m-1}$. It is a subset of $Im(\varphi)$. From lemma \ref{fp} iii/, and from the pigeon hole principle, one can find integers $t_1< t_2$ in $\{0,\dots,m-1\}$ such that $\varphi(\delta_m^{t_1}(U))=\varphi(\delta_m^{t_2}(U))$. If we consider the tuple $\delta_m^{t_1}(U)$ instead of $U$ (they have the same $p$-weight), we get $0<t\leq m-1$ such that $\varphi(U)=\varphi(\delta_m^{t}(U))$.

For each $1\leq i\leq n$, let $u_i=p^{m-t}w_i+v_i$ be the result of the euclidean division of $u_i$ by $p^{m-t}$, set $V=(v_1,\dots,v_n)$, and $W=(w_1,\dots,w_n)$. From lemma \ref{sp} ii/ and the definition of $t$, we have
$$\sum_{i=1}^n v_i\d_i=(p^{m-t}-1)\varphi(U)~;~\sum_{i=1}^n w_i\d_i=(p^{t}-1)\varphi(U).$$
Thus $V\in E_D(m-t)$ and $W\in E_D(t)$. We deduce from their definitions that we have the inequalities $\sigma_p(V)\geq \sigma_p(D,m-t)$, and $\sigma_p(W)\geq \sigma_p(D,t)$. But for each $i$ we have $\sigma_p(u_i)=\sigma_p(v_i)+\sigma_p(w_i)$, and $\sigma_p(U)=\sigma_p(V)+\sigma_p(W)$. From the choice of $U$, we get $\sigma_p(D,m)=\sigma_p(V)+\sigma_p(W)\geq \sigma_p(D,m-t)+\sigma_p(D,t)$, that is
$$\frac{\sigma_p(D,m)}{m}\geq\left(1-\frac{t}{m}\right)\frac{\sigma_p(D,m-t)}{m-t}+\frac{t}{m}\frac{\sigma_p(D,t)}{t}.$$
Thus we get $\frac{\sigma_p(D,m)}{m}\geq \min(\frac{\sigma_p(D,m-t)}{m-t},\frac{\sigma_p(D,t)}{t})$. If $t$ or $m-t$ is greater than $\prod_{j=1}^r D_i$, we use the same process for $V$ or $W$. Thus we get :
$$\frac{\sigma_p(D,m)}{m}\geq \min_{t\leq \prod_{j=1}^r D_i}\left\{ \frac{\sigma_p(D,t)}{t}\right\},$$
and this is the desired result.
\end{proof}

We are ready to define the $p$-density of the set $D$.

\begin{definition}
i/ Let $D,p$ be as above. The {\rm $p$-density of the set $D$} is the rational number
$$\pi_p(D):= \frac{1}{p-1}\min_{m\geq 1}\left\{\frac{\sigma_p(D,m)}{m}\right\}.$$
ii/ The density of an element $U\in E_D(m)$ is $\pi(U):=\frac{\sigma_p(U)}{(p-1)m}$. The element $U$ is {\rm minimal} when $\pi(U)=\pi_p(D)$.
\end{definition}

\subsection{Properties of the $p$-density}

In this subsection we choose $D$ and  $p$ as above.

We come to the general case; we shall give some properties of the $p$-density of a finite subset of $D\subset \N^r$. First we need a definition

\begin{definition}
i/ Let $D\subset \N^r$; we define the {\rm $p$-weight of $D$} as the maximal $p$-weight of the coordinates of its elements, and denote it by $$\sigma_p(D):=\max\{\sigma_p(d_{ij}),~1\leq i\leq n, 1\leq j \leq r\}.$$

ii/ For any $1\leq i\leq n$ we denote by $\sigma_p(\d_i)$ the vector $(\sigma_p(d_{i1}),\dots,\sigma_p(d_{ir}))\in \N^r$. We denote by $\sigma_p(D)$ the set $\{\sigma_p(\d_1),\dots,\sigma_p(\d_n)\}$.
\end{definition}

Now we have :

\begin{lemma}
\label{pp1}
Let $D_1,D_2,D\subset \N^r$, and with none of their elements multiple of $p$.

i/ If $D_1\subset D_2$, then $\pi_p(D_1)\geq \pi_p(D_2)$.

ii/ Assume $D=\{\d_1,\dots,\d_n\}$. Let $\v(v_1,\dots,v_r) \in \R^r$ be such that for any $1\leq i \leq n$ the scalar product $\v\cdot\d_i\leq 1$. Then we have the inequality 
$$\pi_p(D)\geq \sum_{j=1}^r v_j.$$

iii/ Notations are as in ii/.  Let $\v(v_1,\dots,v_r) \in \R^r$ be such that for any $1\leq i \leq n$ the scalar product $\v\cdot\sigma_p(\d_i)\leq 1$. Then we have the inequality 
$$\pi_p(D)\geq \sum_{j=1}^r v_j.$$

iv/ We have the following inequality : $ \pi_p(D)\geq \frac{1}{\sigma_p(D)}$.
\end{lemma}

\begin{proof}
Part i/ follows from the definitions, once we have remarked that for any $m$, $E_{D_1}(m)\subset E_{D_2}(m)$. 

We come to assertion ii/. Let $U=(u_1,\dots,u_n)\in E_D(m)$. For any $1\leq j \leq r$, we get $\sum_{i=1}^n u_id_{ij}=a_j(q-1)$ for some non zero integer $a_j$. From \cite[Proposition 11 iv/]{mm2}, we have $\sigma_p(a(p^m-1))\geq m(p-1)$. Now we have for any $1\leq j \leq r$, the inequalities \begin{equation}
\label{majo}
\sigma_p(\sum_{i=1}^n u_id_{ij})\leq \sum_{i=1}^n \sigma_p(u_id_{ij}) \leq \sum_{i=1}^n \sigma_p(u_i)d_{ij}.
\end{equation}
Thus
$$ \sigma_p(U)=\sum_{i=1}^n \sigma_p(u_i) \geq \sum_{i=1}^n \sigma_p(u_i) \sum_{j=1}^r v_jd_{ij} \geq \sum_{j=1}^r v_j \sum_{i=1}^n \sigma_p(u_i)d_{ij}\geq m(p-1) \sum_{j=1}^r v_j.$$
This is the desired result. Part iii/ can be proven the same way, replacing the inequality \ref{majo} by 
$$\sigma_p(\sum_{i=1}^n u_id_{ij})\leq \sum_{i=1}^n \sigma_p(u_i)\sigma_p(d_{ij}).$$

Assertion iv/ follows from assertion iii/ by taking $\v$ the vector with all coordinates equal to $\frac{1}{d\sigma_p(D)}$.
\end{proof}

We now give a lower bound on $\pi_p(D)$ depending on the convex hull of $D$, which appears in the work of Adolphson and Sperber ({\it cf.} \cite[page 546]{as}) 

\begin{definition}
Let $D\subset \N^r$ as above. Denote by $\Delta(D)$ the convex hull of the points in $D$ and the origin in $\R^r$. We define $\omega(D)$ to be the smallest number such that $\omega(D)\Delta(D)$, the dilation of the polytope $\Delta(D)$ by the factor $\omega(D)$, contains a lattice point with all coordinates positive.
\end{definition}

\begin{proposition}
\label{omega}
We have the inequality $\pi_p(D)\geq \omega(D)$.
\end{proposition}

\begin{proof}
Let $U\in E_D(m)$. Then $\varphi(U)=\frac{1}{p^m-1} \sum_{i=1}^n u_i\d_i$ is a lattice point with all coordinates positive. From the definition of $\omega(D)$, we must have $\sum_{i=1}^n u_i\geq (p^m-1)\omega(D)$. In the same way, we get for any $0\leq k\leq m-1$ $\sum_{i=1}^n \delta_m^k(u_i)\geq (p^m-1)\omega(D)$. Summing over $k$ we get $\sum_{k=0}^{m-1} \sum_{i=1}^n \delta_m^k(u_i)\geq m(p^m-1)\omega(D)$. Now from Lemma \ref{fp} ii/ we have $\sum_{k=0}^{m-1} \delta_m^k(u_i) =\frac{p^m-1}{p-1}\sigma_p(u_i)$, and finally we obtain $\sum_{i=1}^n \sigma_p(u_i)\geq m(p-1)\omega(D)$, that is $\pi(U)\geq \omega(D)$. This is what we claimed. 
\end{proof}

We end this list of general properties giving a lower bound

\begin{lemma}
\label{subsys}
Let $T$ be a non empty subset of $\{1,\dots,r\}$, with cardinality $t$. We denote by $D_T$ the following projection of $D$ on the variables with indices in $T$
$$D_T:=\{\d_i',~ 1\leq i\leq n\},~ \d_i':=(d_{ij})_{j\in T}.$$
We have the inequality $\pi_p(D)\geq \pi_p(D_T)$.
\end{lemma}

\begin{proof}
We just remark that if $S$ is $\{1,\dots,r\}\backslash T$, then for each $m\geq 1$ we have $E_D(m)=E_{D_T}(m)\cap E_{D_S}(m)$ since this corresponds to cutting the system of modular equations in two subsystems. Now the lemma follows readily from the definitions.
\end{proof}

\subsection{The case $r=1$}
We begin by considering the case $D=\{d\}$, where $d>1$ is an integer prime to $p$ (from Remark 1.1).

\begin{definition}
For any rational $x\in \Q$, we denote by $<x>=x-[x]$ its fractional part. Let $\ell$ be the order of $p$ in the multiplicative group $(\Z/d\Z)^\times$. For any integer $1\leq a\leq d-1$, we define
$$\tau_d(a)=\frac{1}{\ell}\sum_{i=0}^{\ell-1} <\frac{ap^i}{d}>.$$
\end{definition}

\begin{remark}
The rational number $\tau(a)$ is well known. Let $m$ be an integer such that $p^m\equiv 1 ~[d]$. If $\omega$ denotes the Teichm\"uller character of the finite field $\F_{p^m}$, then from Stickelberger's theorem, $\tau(a)$ is the $p^m$-valuation of the Gauss sum over $\F_{p^m}$ associated to the character $\omega^{-a\frac{p^m-1}{d}}$.
\end{remark}

\begin{proposition}
Assume $D= \{d\}$, $d\neq 1$. We have the equality 
$$\pi_p(D)=\min_{1\leq a \leq d-1} \{\tau_d(a)\}.$$
\end{proposition}

\begin{proof}
Let $u\in E_D(m)$. We have $ud=a(p^m-1)$. If we set $d_m:=\gcd(d,p^m-1)$, we must have $u=a'(p^m-1)/d_m$ and $a=a'd/d_m$ for some $a'\in\{1,\dots,d_m-1\}$. Now from \cite[Theorem 11.2.7]{bew}, we have 
$$\sigma_p(u)=(p-1)\sum_{i=0}^{m-1} <\frac{a'p^i}{d_m}>.$$ 
If $\ell_m$ denotes the order of $p$ in the multiplicative group $(\Z/d_m\Z)^\times$, we must have $\ell_m|m$, and we get $\sigma_p(u)=(p-1)\frac{m}{\ell_m}\sum_{i=0}^{\ell_m-1} <\frac{a'p^i}{d_m}>$. On the other hand we have $\ell_m|\ell$, and we get
$$\sigma_p(u)=(p-1)\frac{m}{\ell}\sum_{i=0}^{\ell-1} <\frac{a'p^i}{d_m}>=(p-1)\frac{m}{\ell}\sum_{i=0}^{\ell-1} <\frac{ap^i}{d}>=m(p-1)\tau_d(a).$$
The result follows from the definition of $\pi_p(D)$.
\end{proof}

We have the following symmetry: $\tau_d(a)+\tau_d(d-a)=1$. Thus for any $a$ one of these numbers is less than $\frac{1}{2}$. From Lemma \ref{pp1} we get bounds in the one dimensional case

\begin{corollary}
\label{demi}
Assume $D\subset \N\backslash\{0\}$, such that $D\neq \{1\}$ and $D$ contains no multiple of $p$; we have the inequality $\frac{1}{\sigma_p(D)}\leq \pi_p(D)\leq \frac{1}{2}$.
\end{corollary}

 \begin{remark}
 When $D=\{d\}$ and $p$ is semi primitive modulo $d$, i.e. when some power of $p$ is $-1$ modulo $d$, the right inequality in the above corollary is an equality.
 \end{remark}
   
We end this section with a practical result for the numerical determination of the density.

\begin{lemma}
\label{length}
Let $U\in E_D(m)$ with density $\pi$, such that $\Phi(U)$ contains exactly $m$ elements. If $d=\max D$, we have
$$m\leq 2d\pi -1.$$
\end{lemma}

\begin{proof}
From Lemma \ref{sp} i/, we have 
$$\sum_{i=1}^{n} \sigma_p(u_i)d_i=(p-1)\sum_{k=0}^{m-1} \varphi(\delta_m^k(U)).$$ 
The numbers $\varphi(\delta_m^k(U))$ are the elements of $\Phi(U)$. They are positive and pairwise distinct. Thus the right hand side of the inequality is greater than or equal to $\sum_{k=1}^{m} k=m(m+1)/2$.

On the other hand, we get $\sum_{i=1}^{n} \sigma_p(u_i)d_i\leq d\sum_{i=1}^{n} \sigma_p(u_i)=d\sigma_p(U)=dm(p-1)\pi$. Combining these inequalities gives the desired result.
\end{proof}

We end this section with an example that we shall use in the last section.

\begin{lemma}
\label{calc}
Fix a prime $p$ and let $D$ be the set of prime to $p$ integers in $\{1,\dots,d\}$. For any $p^n-1\leq d< 2p^n-1$, we have $\pi_p(D)=\frac{1}{n(p-1)}$.
\end{lemma}

\begin{proof}
First note that $p^n-1$ is the least integer $k$ such that $\sigma_p(k)=n(p-1)$, and $2p^n-1$ is the least integer $k$ such that $\sigma_p(k)=n(p-1)+1$. Thus $\sigma_p(D)=n(p-1)$, and from Corollary \ref{demi}, we have $\pi_p(D)\geq \frac{1}{n(p-1)}$.
On the other hand for any $i\geq 1$ we have $(1+p^n+\dots+p^{(i-1)n})(p^n-1)\equiv 0$ mod $p^{in}-1$. Since $\sigma_p(1+p^n+\dots+p^{(i-1)n})=i$, we get $\pi_p(D)\leq \frac{1}{n(p-1)}$. This shows the assertion.
\end{proof}

\medskip

\section{Valuation of exponential sums}

In this section we fix a polynomial $f\in \overline{\F}_p[x_1,\dots,x_r]$ such that all the variables actually appear in $f$ (else we can reduce to lower dimension), and we set $D:=D(f)$ to be the set of exponents of $f$ in $\N^r$.

Assume $f\in \F_q[x_1,\dots,x_r]$. Let $\psi$ denote a non trivial additive character of $\F_q$. The exponential sum associated to $f$ over $\F_q$ is the sum
$$S_q(f):=\sum_{(x_1,\dots,x_r)\in \ma{F}_q^r} \psi(f(x_1,\dots,x_r)).$$

Our main result is the following estimate for the $q$-adic valuation of the sum $S(f)$

\begin{theorem}
\label{vse}
Let $p$, $f$, $D$ be as above. If $v_q$ denotes the $q$-adic valuation, we have the inequality
$$v_q(S_q(f))\geq \pi_p(D).$$
Moreover this inequality is optimal in the sense that for $p$ and $D$ fixed, there exists a power $q$ of $p$ and a polynomial $f$ over $\F_q$ with its exponents in $D$ such that it is an equality.
\end{theorem}

\begin{proof}
The inequality is an easy consequence of \cite[Theorem 8]{mm2}: if $q=p^m$, the number $L$ in this theorem is exactly $\sigma_D(m)\geq m(p-1)\pi_p(D)$. Since the number $\pi$ in the same theorem has valuation $v_q(\pi)=\frac{1}{m(p-1)}$, we get the result.

The second assertion comes from \cite[Theorem 9]{mm2}: from Proposition \ref{princ} we can choose $m\geq 1$ such that $\sigma_D(m)=m(p-1)\pi_p(D)$. Thus the number $L$ is exactly $m(p-1)\pi_p(D)$, and there is at least one polynomial $f$ over $\F_{p^m}$ with its coefficients in $D$ such that $S(f)$ is of exact valuation $\pi_p(D)$.
\end{proof}

\begin{remark}
From Proposition \ref{omega}, we get a better bound than \cite[Theorem 1.2]{as}. The price to pay is that our bound depends on $p$, and the one in \cite{as} does not.
\end{remark}

As usual, we deduce a theorem in the Chevalley-Warning, Ax-Katz family from the bound above. We first need some notations.

\medskip

Let $f_1,\dots,f_s \in \F_q[x_1,\dots,x_r]$ be polynomials in $r$ variables over $\F_q$, an extension of $\F_p$. For any $1\leq i\leq s$, let $D_i$ be the set of exponents of $f_i$ in $\N^r$. Let $X$ be the variety (defined over $\F_q$) defined by the common vanishing of the $f_i$, $1\leq i\leq s$. We denote by $D(X)$ the subset of $\N^{r+s}$ which is the union over $1\leq i\leq s$ of the $D_i\times\{0,\dots,0,1,0,\dots,0\}$, where the $1$ is at the $i$-th place. Then we have

\begin{theorem}
The number of $\F_q$-rational points of $X$, $N_q(X)$ satisfies
$$v_q(N_q(X))\geq \pi_p(D(X))-s.$$
\end{theorem}

\begin{proof}
It is very classical. Consider the polynomial 
$$g(x_1,\dots,x_r,y_1,\dots,y_s):=\sum_{i=1}^s y_if_i(x_1,\dots,x_r)~\in~\F_q[x_1,\dots,x_r,y_1,\dots,y_s];$$
 from the orthogonality relations on additive characters, we have $S_q(g)=q^sN_q(X)$. Now the set of exponents of $g$ is exactly $D(X)$, and the result follows from Theorem \ref{vse}.
\end{proof}

\section{The generic first slope of Artin Schreier curves}

Usually, an Artin Schreier curve is a cyclic covering of the projective line of degree $p$ over a field of characteristic $p$. In other words, it is a curve having an affine model with equation $y^p-y=f(x)$, where $f\in \F_q(x)$ is a rational function. If $f$ has at least two poles, Deuring Shafarevic formula tells that the $p$-rank of the jacobian of $C$ is positive. In other words, the first segment of the Newton polygon of the numerator of its zeta function is horizontal. 

We assume that $f$ has only one pole, actually that $f$ is a polynomial. It is well known in this case that the Newton polygon of the numerator of its zeta function is the dilation with factor $p-1$ of the Newton polygon of the $L$-function associated to the sums
$$S_k(f)=\sum_{x\in \ma{F}_{q^k}} \psi(\Tr_{\ma{F}_{q^k}/\ma{F}_q}(f(x))),~k\geq 1$$
for any non trivial additive character $\psi$ of $\F_{q^k}$. Thus we shall consider the $q$-adic Newton polygon of this $L$-function.

\medskip

We explain precisely what we mean by ``generic first slope". A general result concerning Newton polygons is {\it Grothendieck's specialization theorem}. In order to quote it, let us recall some results about crystals. Let $\L_\psi$ denote the {\it Artin Schreier crystal}; this is an overconvergent $F$-isocrystal over $\A^1$, and for any polynomial $f\in \F_q[x]$ of degree $d$, we have an overconvergent $F$-isocrystal $f^*\L_\psi$ with
$$L(f,T)=\det\left(1-T\phi_c|H^1_{\rm rig,c}(\A^1/K,f^*\L_\psi)\right).$$
Let us parametrize the set of monic polynomials with their exponents in $D=\{d_1,\dots,d_n\}$ by the affine space $\A_D$ of dimension $n-1$, associating the point $(a_1,\dots,a_{n-1})$ to the polynomial $f(x)=x^{d_n}+a_{n-1}x^{d_{n-1}}+\dots+a_1x^{d_1}$; we can consider the family of overconvergent $F$-isocrystals $f^*\L_\psi$. Now for a family of $F$-crystals $(\M,F)$ of rank $r$ over a $\F_p$-algebra $\AA$, we have Grothendieck's specialization theorem ({\it cf.} \cite{ka2} Corollary 2.3.2)

\medskip

{\it Let $P$ be the graph of a continuous $\R$-valued function on $[0,r]$ which is linear between successive integers. The set of points in $\Spec(\AA)$ at which the Newton polygon of $(\M,F)$ lies above $P$ is Zariski closed, and is locally on $\Spec(\AA)$ the zero-set of a finitely generated ideal.}

\medskip

In other words, this theorem means that when $f$ runs over monic polynomials with exponents in $D$ over $\overline{\F}_p$, there is a Zariski dense open subset $U_{D,p}$ (the {\it open stratum}) of the (affine) space of these polynomials, and a {\it generic Newton polygon} $GNP(D,p)$ such that for any $f\in U_{D,p}(\F_q)$, $NP_q(f)=GNP(D,p)$, and $NP_q(f)\preceq GNP(D,p)$ for any $f\in \A_D(\F_q)$ (where $NP\preceq NP'$ means $NP$ lies above $NP'$). 

\medskip

\begin{definition}
\label{first}
The {\rm generic first slope} for the family of monic polynomials with their exponents in $D$ and coefficients in $\overline{\F}_p$ is the first slope of the generic polynomial $GNP(D,p)$. We denote it by $s_1(D,p)$.
\end{definition}

\medskip

Another application of $p$-density is the following

\begin{theorem}
\label{first}
We have the equality $s_1(D,p)=\pi_p(D)$. 
\end{theorem}

\begin{proof}
The inequality $s_1(D,p)\geq \pi_p(D)$ follows from Theorem \ref{vse} and the argument in \cite[Introduction]{ax}: for any $f\in \A_D(\F_q)$, and any $k\geq 1$, we have $v_q(S_k(f))\geq k \pi_p(D)$, thus every reciprocal root of the function $L(f,T)$ has $q$-adic valuation greater than $\pi_p(D)$. Thus for any $f\in \A_D$, the first slope of the $q$-adic Newton polygon of $L(f,T)$ is greater than the $p$-density of $D$.

We show the equality. Again from Theorem \ref{vse}, there exists some extension $\F_q$ of $\F_p$, and some polynomial in $\A_D(\F_q)$ such that $v_q(S_k(f))= \pi_p(D)$. Thus at least one of the reciprocal roots of the polynomial $L(f,T)$ has $q$-adic valuation $\pi_p(D)$, and its Newton polygon has first slope equal to $\pi_p(D)$.
\end{proof}

We apply this theorem to construct families of supersingular Artin-Schreier curves. A curve is supersingular exactly when its Newton polygon consists of a segment of (horizontal) length $2g$ and slope $\frac{1}{2}$. From the above Theorem we get :

\begin{corollary}
\label{supersing}
i/ The family of Artin Schreier curves $y^p-y=f(x)$, $f\in \A_D(\F_p)$, is supersingular exactly when we have $\pi_p(D)=\frac{1}{2}$.

ii/ (cf. \cite{vv} for $p=2$) When $D$ is a finite subset of $\{1,p^i+1\}_{i\geq 0}$, the family above is supersingular over $\F_p$.
\end{corollary}

\begin{remark} 
From lemma \ref{length}, we have $\pi_p(D)=\frac{1}{2}$ exactly when for any $1\leq m\leq \max D-1$, and any $U\in E_D(m)$, we have $\pi(U)\geq \frac{1}{2}$. Thus for fields of small characteristic, and polynomials of small degrees (actually the ones most likely to produce supersingular curves), an exhaustive computer search allows to prove a family supersingular.
\end{remark} 

\begin{example}
Apart from the families already given in the corollary above, the following families are supersingular (the last column gives the reference to prior work when the family is alredy known)

$$\begin{array}{ccc}

 p \qquad & D & \mbox{\rm Reference}\\
\\
2 \qquad & \{11, 3, 1 \} & \mbox{\rm \cite{sz2}} \\
\\

2 \qquad & \{13, 3, 1 \} & \mbox{\rm \cite{sz2}} \\
\\
3 \qquad & \{7, 2, 1 \} & \mbox{\rm \cite{zhu}} \\
\\
3 \qquad & \{14, 2, 1 \} &  \\
\\
5 \qquad & \{7, 1 \} & \mbox{\rm \cite{zhu}} \\
\\
7 \qquad & \{5, 2 \} & \mbox{\rm \cite{zhu}} \\
\\
\end{array}$$
Note that the only new example above is not very surprising, since the curve $y^3-y=x^{14}+ax^2+bx$ is covered by the curve $y^3-y=x^{28}+ax^4+bx^2$, which is itself supersingular by Corollary \ref{supersing}.
\end{example}

Finally we give the generic first slope of families of Artin-Schreier curves with given genus and $p$-rank $0$. The following result is an immediate consequence of Theorem \ref{first} and Lemma \ref{calc}.

\begin{proposition}
The generic first slope of the family of Artin Schreier curves with $p$-rank $0$ and genus $g=\frac{(p-1)(d-1)}{2}$ is $\frac{1}{n(p-1)}$ when $p^n-1\leq d< 2p^n-1$.
\end{proposition}

\begin{remark} 
Note that for $p=2$ we get that the generic first slope of Artin Schreier curves of $2$-rank $0$ and genus $g$ in $\{2^{n-1}-1,\dots,2^n-2\}$ is $\frac{1}{n}$. Compare with \cite[Theorem 1.1]{sz1}.
\end{remark}

\end{document}